\documentclass[a4paper,11pt,twoside,reqno]{amsart}

\usepackage[utf8]{inputenc}
\usepackage[plainpages=false,pdfpagelabels=true]{hyperref}
\usepackage{amssymb,amsthm}
\usepackage[margin=1in]{geometry}
\usepackage{slashed}

\newtheorem{Satz}{Theorem}[section]
\newtheorem{Prop}[Satz]{Proposition}

\theoremstyle{definition}
\newtheorem{Dfn}[Satz]{Definition}
\newtheorem{Bem}[Satz]{Remark}

\newcommand{\tr}{\operatorname{Tr}}
\newcommand{\Ric}{\operatorname{Ric}}
\newcommand{\Ein}{\operatorname{Ein}}

\newcommand{\Scal}{\operatorname{Scal}}

\newcommand{\grad}{\operatorname{grad}}

\newcommand{\sff}{\mathrm{I\!I}}

\newcommand{\dv}{\,\mathrm{d}v}   

\newcommand{\R}{\mathbb{R}}
\allowdisplaybreaks[1]

\renewcommand{\epsilon}{\varepsilon}
\numberwithin{equation}{section}

\usepackage{pifont}
\usepackage{booktabs}
\usepackage{empheq}
\usepackage{tensor}
\usepackage[inline]{enumitem}
\usepackage{mathrsfs}
\setenumerate{font=\upshape,label=(\roman*),itemsep=0.1cm}

\newcommand{\diff}{\mathrm{d}}
\newcommand{\Y}{\mathrm{Y}}
\newcommand{\id}{\mathrm{Id}}
\newcommand{\iprod}{\mathbin{\lrcorner}}
\newcommand{\Fer}{\mathscr{F}} 
\newcommand{\Hig}{\mathscr{H}} 
\newcommand{\dirac}{\slashed{\mathrm{D}}{}}
\newcommand{\g}{\mathfrak{g}}

\newcommand{\comment}[1]{}

\newcommand{\ferdirac}[1]{\dirac_{#1}}

\renewcommand{\Re}{\mathrm{Re}}
\renewcommand{\Im}{\mathrm{Im}}

\DeclareMathOperator{\SO}{\mathbf{SO}}

\DeclareMathOperator{\Ad}{Ad}

\newcommand{\normal}{\mathbf{n}}

\usepackage[dvipsnames]{xcolor}

\title[The energy-momentum tensor of the Standard Model]{The energy-momentum tensor of the Standard Model with applications to energy conditions}

\date{\today}

\author{Volker Branding}
\address{University of Rostock, Institute of Mathematics\\
Ulmenstraße 69, 18057 Rostock, Germany\\}
\email{volker.branding@uni-rostock.de}

\author{Adam Lindstr\"om}
\address{University of Vienna, Faculty of Mathematics\\
Oskar-Morgenstern-Platz 1, 1090 Vienna, Austria\\}
\email{adam.lindstroem@univie.ac.at}

\author{Marko Sobak}
\address{University of Vienna, Faculty of Mathematics\\
Oskar-Morgenstern-Platz 1, 1090 Vienna, Austria\\}
\email{marko.sobak@univie.ac.at}

\thanks{The authors gratefully acknowledge the support of the Austrian Science Fund (FWF) through the projects "Geometric Analysis of Biwave Maps" (DOI: 10.55776/P34853) and 
"The Standard Model as a Geometric Variational Problem" (DOI: 10.55776/P36862).
The second and third named authors also thank the University of Rostock for the hospitality during their research visits where parts of this article were written.
}

\begin{document}
\begin{abstract}
The Standard Model of elementary particle physics is one of the most successful models of contemporary physics,
its predictions being in full agreement with experiments. In this manuscript we consider the Lagrangian of the Standard Model as a geometric variational problem on a globally hyperbolic manifold and derive the associated energy-momentum tensor in a geometric invariant way.  
As an application, we investigate the validity of various energy conditions that arise in general relativity.
\end{abstract} 

\maketitle

\section{Introduction}
The Standard Model of elementary particle physics is one of the greatest successes in modern theoretical physics.
Its mathematical structure employs gauge theory on globally hyperbolic manifolds.
More precisely, we use the language of principal fiber bundles to model the particle content of the theory, 
and to incorporate the dynamics of the system we put the governing equations on a globally hyperbolic manifold as this is the natural setup for wave equations.

The central object of the Standard Model is a Lagrangian whose precise structure is fixed by demanding a number of invariances demanded by physics.
In order to provide the mathematical background of the Standard Model Lagrangian, it is convenient
to consider the three sectors contributing to its Lagrangian
\begin{align}
\label{eq:lagrangian-intro}
L^{\mathrm{SM}}=L^{\mathrm{YM}}+L^{\mathrm{Higgs}}+L^{\mathrm{Dirac}},    
\end{align}
which are the Yang-Mills, Higgs, and Dirac(-Yukawa) sector, respectively.

The relevant mathematical background for defining the Yang-Mills sector is covered in
\cite{MR1079726}, see also \cite{BAUM}.
For spin geometry in the Riemannian setting we refer to \cite{MR1031992}, for spin geometry on 
Lorentzian manifolds one may consult \cite{MR701244}.
A modern mathematical introduction to the Standard Model is provided by the recent book of Hamilton \cite{MR3837560}. 
A study of the energy functional of the Standard Model in the case of a four-dimensional Riemannian domain is carried out in the article of Parker \cite{MR677998}. For the necessary background on globally hyperbolic manifolds we refer to
the books \cite{MR424186}, \cite{MR2527641}. 

Currently, there are only few mathematical references available that provide existence results
for the critical points of the complete Standard Model Lagrangian \eqref{eq:lagrangian-intro}.
In this regard we want to mention the seminal result of Choquet-Bruhat and Christodoulou \cite{MR654209}
where an existence result on \((3+1)\)-dimensional Minkowski space assuming small initial data
using conformal methods was established. This result was extended recently by the first and the third author in \cite{bs25} assuming that the Standard Model is considered on expanding space-times and small initial data. However, both results require to employ a conformal potential for the Higgs field instead of the famous Mexican hat potential from physics.

Given a model in physics that is defined by a Lagrangian as \eqref{eq:lagrangian-intro}
one can ask how this model behaves when the metric on the underlying manifold changes.
This naturally leads to the notion of the \emph{energy-momentum tensor}, often also called the stress-energy-tensor, which arises from the variation of the Lagrangian of a theory
with respect to the metric of the manifold it is defined on. While the energy-momentum tensor is an object of individual mathematical interest, it is of great importance in particular when coupling the theory to the Einstein equations
of general relativity.

Let us a give a non-exhaustive list of results on the energy-momentum tensor for parts of the Standard Model which are connected to the content of this manuscript.
\begin{enumerate}
    \item The role of the energy-momentum tensor in classical gauge theories in physics is highlighted in \cite{MR3567581} and \cite{MR2029901}, see also the references cited therein.
    \item The energy-momentum tensor for the Dirac sector was studied in the Riemannian case in \cite{MR1158762} which was then extended to arbitrary signatures in \cite{MR2121740},
    see also \cite{MR4034775} for a closely related problem.
    \item Concerning the energy-momentum tensor for the Yang-Mills action we want to mention the articles \cite{MR2453667}, \cite{MR2795324}, \cite{MR654088} which all focus on the case of a Riemannian domain.
    \item In their seminal work on the existence of the Yang-Mills-Higgs system on four-dimensional Minkowski space Eardley and Moncrief employed the associated energy-momentum tensor to derive energy estimates, see \cite{MR649158},\cite{MR649159}.
    \item In \cite{MR4038549} an existence result for the Yang-Mills-Higgs system on four-dimensional Minkowski space for large initial data is achieved, where the energy-momentum tensor plays a centrole role in setting up the necessary energy estimates.
\end{enumerate}

One of the goals of this article is to understand if the energy-momentum tensor derived from the Lagrangian of the Standard Model satisfies or violates certain energy conditions that are of great importance in mathematical general relativity. Through the singularity theorems of Hawking and Penrose one expects that a violation of one these energy conditions leads to the formation of a space-time singularity so that it is important to understand the precise structure of the energy-momentum tensor in this regard.

For an overview on different energy conditions as they appear in quantum field theory and general relativity we refer to the reviews \cite{MR4190325} and \cite{MR4092298}.

\par\medskip
This article is organized as follows:
In Section 2, we fix necessary notation required and briefly recall the geometric structure of the Standard Model Lagrangian.
Section 3 then carefully derives the energy-momentum tensor by varying the corresponding energy functional with respect to the metric both for the standard Higgs potential as well as its conformal version.
In addition, we show by an explicit calculation that the energy-momentum tensor is conserved at critical points 
and discuss the influence of the Higgs potential on its properties.
We then apply these results in Section 4 to investigate if the energy-momentum tensor
of the Standard Model satisfies or violates various energy conditions which come up in general relativity.

\section{Geometric structure of the Standard Model}
\label{sec-standard-model}

Throughout this manuscript, we assume $(M,g)$ to be an oriented and time-oriented $m$-dimensional spin Lorentzian manifold of signature $(-+\cdots+)$. We use the same conventions as in \cite{bs25} for forms, inner products, curvature, signs of operators, etc., so we do not restate them here. However, we briefly recall the structure of the Standard Model Lagrangian.

\subsection{Bosons (Yang-Mills-Higgs sector)}

\textit{Bosons} are the so-called \textit{force-carrier particles}, as they mediate forces between all other particles.
These forces include the electromagnetic force, the weak interaction, and the strong interaction. 

Let $G$ be a compact Lie group equipped with an Ad-invariant positive-definite inner product on its Lie algebra $\g$.
We would like to note that for the Standard Model, the structure group is explicitly (see \cite{MR3837560})
\begin{equation*}
    G = (\mathbf{SU}(3) \times \mathbf{SU}(2) \times \mathbf{U}(1))/\mathbb{Z}_6,
\end{equation*}
but the results of this paper apply for any compact Lie group as above.

Let $G \to P \xrightarrow\pi M$ be a principal bundle equipped with a connection $\omega \in \Omega^1(P,\g)$, and let $F_\omega \in \Omega^2(M, \Ad P)$ be its curvature form.
The connection $\omega$ encodes information about \emph{bosons}.
Let also $W$ be a complex linear space equipped with a Hermitian inner product and let $\rho : G \to \mathbf{U}(W)$ be a complex representation. We consider the associated bundle $\Hig = P \times_\rho W$ whose sections are called \textit{Higgs fields}.
The connection $\omega$ induces a covariant derivative on $\Hig$ that is denoted by $\nabla_{\omega}$.
Since the representation $\rho$ is unitary, the inner product on $W$ also induces a metric on the Higgs field bundle $\Hig$.
The \textit{Yang-Mills-Higgs (bosonic) Lagrangian density} is given by 
\begin{align*}
    (\omega, \Phi) \mapsto &-\left(|F_\omega|^2 + |\nabla_{\omega} \Phi|^2 + U(\Phi)\right)
    \\
    =& -\frac12 \langle (F_\omega)_{\mu\nu}, (F_\omega)^{\mu\nu}\rangle - \langle (\nabla_\omega\Phi)_\mu, (\nabla_\omega\Phi)^\mu \rangle - U(\Phi),
\end{align*}
where $U: \Hig \to \R$ is a smooth $G$-invariant function.

\begin{Bem}
The Standard Model utilizes the so-called \textit{Mexican hat potential}, given by
\begin{equation*}
     U_{\text{mex}}(\Phi) = - \mu |\Phi|^2 + \frac{\lambda}{2} |\Phi|^4
\end{equation*}
where $\mu,\lambda \geq 0$ are constants.
However, in principle, the potential $U$ could also have additional dependences on the spacetime.
Here, it is most reasonable to have this dependence driven by the spacetime metric $g$, so that it is natural to also consider potentials $U(g, \Phi)$.
In this context a more explicit example is the \textit{conformal potential}
\begin{equation}\label{eq-conformal-potential}
    U_{\text{conf}}(g, \Phi) = \frac{m-2}{4(m-1)} \Scal_g |\Phi|^2 + \frac{\lambda}{2} |\Phi|^4,
\end{equation}
where $\Scal_g$ is the scalar curvature of $(M,g)$ and $\lambda\geq 0$ is a non-negative function on $M$.
This potential has the benefit that it renders the Euler-Lagrange equations for the matter fields conformally invariant.

In the last section of this article we will be investigating the validity of various energy condition arising in general relativity.
One of these conditions, i.e. the weak energy condition, demands the potential $U$ to be non-negative. 
This can be achieved by considering the following version of the Mexican hat potential 
\begin{equation}\label{eq-mexican-hat-potential}
    U_{\text{mex}}(\Phi) = \frac{\lambda}{2}\left(|\Phi|^2 - \frac{\mu}{\lambda} \right)^2.
\end{equation}
On the other hand, the conformal potential (\ref{eq-conformal-potential}) is already non-negative if $\Scal_g \geq 0$ and otherwise one should replace it by
\begin{equation*}
    (g, \Phi) \mapsto \frac{\lambda}{2} \left(|\Phi|^2 + \frac{m-2}{4(m-1)\lambda}\Scal_g\right)^2.
\end{equation*}
Note that these modifications do not affect the Euler-Lagrange equation for the matter fields, but they do affect the energy-momentum tensor and hence the Einstein field equations. 
In particular, in the case of the conformal potential, one also needs to compute the variation of $\Scal^2_g$. For simplicity, we will only consider the simpler (\ref{eq-conformal-potential}) in this context.
\end{Bem}

\subsection{Fermions (Dirac-Yukawa sector)}

\textit{Fermions} are the so-called \textit{matter particles}.
They can further be subdivided into \textit{quarks} and \textit{leptons}.
Mathematically, elementary fermionic particles are modelled as twisted chiral spinors. 
Note that the notion of chirality only makes sense in even dimensions, which is of course the case for the Standard Model as it is physically built on a four-dimensional spacetime. However, one can formulate the Dirac sector on spacetimes of arbitrary dimension, while only the Yukawa (mass) sector requires the chiral structure. We will elaborate on this below.

Let $V$ be a complex linear space equipped with a Hermitian inner product.
Let $\chi : G \to \mathbf{U}(V)$ be a complex representation and define the associated bundle $\mathscr{S} = P \times_\chi V$. We consider the twisted spinor bundle 
\begin{equation*}
\Fer = \Sigma M \otimes \mathscr{S}.
\end{equation*}
The connection $\omega$ induces a covariant derivative on $\mathscr{S}$, which together with the spinor covariant derivative induces a \textit{twisted spinor covariant derivative} which will be denoted by $\nabla_\omega$.
The \textit{twisted Dirac operator} is the map
\begin{equation*}
    \dirac_{\omega} : \Gamma(\Fer) \to \Gamma(\Fer), \qquad 
    \dirac_{\omega}\Psi = \eta^{\mu\nu} e_\mu \cdot (\nabla_{\omega} \Psi)(e_\nu). 
\end{equation*}
Here, Clifford multiplication acts only on the spinorial factor,
and we use the sign convention
\begin{align}
   \label{eq-clifford-rel}
X\cdot Y\cdot\psi+Y\cdot X\cdot\psi=-2g(X,Y)\psi.
\end{align}
Note that \(i\dirac_\omega\) is self-adjoint with respect to the \(L^2\)-norm, since Clifford multiplication is symmetric in our convention, i.e. 
\begin{align}
\label{eq-clifford-mult}
\langle X\cdot\psi,\xi\rangle=\langle\psi,X\cdot\xi\rangle
\end{align}
for all \(\psi,\xi\in\Gamma(\Fer)\) and \(X\in TM\).
The twisted Dirac operator \(\dirac_{\omega}\) satisfies the
Weitzenböck formula
\begin{align}
\label{eq:weitzenboeck-dirac}
\dirac_{\omega}^2\Psi= -\Box_\omega\Psi + \frac{1}{4}\Scal_g \Psi + \chi_*(F_\omega)\cdot \Psi   
\end{align}
for all $\Psi \in \Gamma(\Fer)$.

Now, assuming that $M$ is even-dimensional, the spinor bundle splits into chiral subbundles $\Sigma M = \Sigma_+ M \oplus \Sigma_- M$.
In this case, we assume that the representation $\chi$ splits such that $V = V_+ \oplus V_-$ is an orthogonal decomposition and $\chi = \chi_+ \oplus \chi_-$ where $\chi_\pm : G \to \mathbf{U}(V_\pm)$ are representations (in general non-isomorphic).
We also define the associated bundles $\mathscr{S}_\pm = P \times_{\chi_\pm} V_\pm$, as well as the \textit{twisted chiral/mixed spinor bundles}, respectively, as 
\begin{equation*}
    \Fer_+ = (\Sigma_+M \otimes \mathscr{S}_+) \oplus (\Sigma_-M \otimes \mathscr{S}_-), \qquad \Fer_- = (\Sigma_+M \otimes \mathscr{S}_-) \oplus (\Sigma_-M \otimes \mathscr{S}_+).
\end{equation*}
Sections $\Psi \in \Gamma(\Fer_+)$ are called \textit{twisted chiral spinors}.
The twisted Dirac operator restricts to a mapping
\begin{equation*}
    \dirac_{\omega} : \Gamma(\Fer_\pm) \to \Gamma(\Fer_\mp).
\end{equation*}

Now note that $\Fer = \Fer_+ \oplus \Fer_-$ is a null decomposition, since $\Sigma_\pm M$ are null and $\mathscr{S}_\pm$ are orthogonal.
In particular, we cannot supplement the energy density with a mass term of the form $m^2 |\Psi|^2$ as such a term is identically zero for $\Psi \in \Gamma(\Fer_+)$. 
The solution to this problem is again provided by the Higgs field, via the so-called Yukawa coupling, which is a mapping taking Higgs fields to couplings between $\Fer_\pm$, i.e.\
\begin{equation*}
    \Gamma(\Hig) \ni \Phi \mapsto \Y_\Phi : \Gamma(\Fer_\pm) \to \Gamma(\Fer_\mp),
\end{equation*}
such that $i\Y_\Phi$ is self-adjoint (see \cite{bs25} for more details).
Note that these considerations apply for even-dimensional spacetimes of any dimension.

On the other hand, when the spacetime is odd-dimensional, there is no notion of chirality, and we instead consider the sections of the full space $\Fer$, together with a self-adjoint "mass" map $i\Y_\Phi$ taking sections of $\mathscr{F}$ to sections of $\mathscr{F}$. E.g.\ one can simply define $i\Y_\Phi\Psi = m\Psi$ for $m \in \R$.

Thus, for fixed $\omega$ and $\Phi$, the total Lagrangian density for fermions is given by
\begin{equation*}
    \Psi \mapsto  -\Re \langle \Psi, i(\dirac_{\omega} + \Y_\Phi) \Psi \rangle.
\end{equation*}

\subsection{Lagrangian and Euler-Lagrange equations}
The total Lagrangian of the Standard Model is thus given by
\begin{equation}\label{eq-sm-lagrangian}
    L^{\mathrm{SM}}: (\omega,\Phi,\Psi) \mapsto 
    -\left(\vert F_\omega \vert^2  + \vert \nabla_{\omega} \Phi \vert^2 + U(g, \Phi) + 
    \Re\langle \Psi, i(\dirac_{\omega} + \Y_\Phi) \Psi \rangle\right),
\end{equation}
where $\omega \in \Omega^1(P,\g)$ is a connection, $\Phi \in \Gamma(\Hig)$, $\Psi \in \Gamma(\Fer_+)$, and $U: M \times \Hig \to \R$ is a smooth $G$-invariant function.

The Standard Model (SM) field equations are obtained by varying the Lagrangian $L^{\mathrm{SM}}$ with respect to $(\omega, \Phi, \Psi)$.
The variation is well-known in the literature, we refer to \cite{bs25} for the details.
The resulting equations are called the \emph{SM equations} and are explicitly given by
\begin{subequations}
\begin{empheq}[left=\empheqlbrace]{align}
    \diff_\omega^\ast  F_\omega - \mathfrak{J}_1[\Phi] - \mathfrak{J}_2[\Psi] &= 0 , \label{eq-yang-mills}\\[0.2cm]
    \Box_{\omega} \Phi - \tfrac12 \grad U_\Phi - \mathfrak{J}_3[\Psi] &= 0 , \label{eq-higgs}\\[0.2cm]
    \dirac_{\omega} \Psi + \Y_\Phi\Psi &= 0, \label{eq-dirac}
\end{empheq}
\end{subequations}
where $\omega \in \Omega^1(P,\g)$ is a connection, $\Phi \in \Gamma(\Hig)$, $\Psi \in \Gamma(\Fer_+)$, and $U: \Hig \to \R$ is a smooth $G$-invariant function. Moreover, the currents are given by
\begin{align*}
    \mathfrak{J}_1[\Phi] &= -\Re \langle \nabla_{\omega}\Phi \otimes \rho_*\Phi \rangle = -\Re \langle (\nabla_{\omega}\Phi)_\mu, \, \rho_*(\xi_a)\Phi \rangle \, e^\mu \otimes \xi_a,\\[0.1cm]
    \mathfrak{J}_2[\Psi] &= \tfrac12\Im \langle \id\cdot\Psi \otimes \chi_*\Psi \rangle = \tfrac12 \Im \langle e_\mu \cdot \Psi, \, \chi_*(\xi_a)\Psi \rangle\, e^\mu \otimes \xi_a,\\[0.1cm]
    \mathfrak{J}_3[\Psi] &= \langle \Psi, i\Y^-  \Psi \rangle = \tfrac{1}{2} \langle \Psi, (i\Y_{W_k} - \Y_{iW_k}) \Psi \rangle \, W_k.
\end{align*}
Here, $e_\mu$ is a semi-orthonormal frame for $\SO^+(TM)$, $\xi_a$ is an orthonormal frame for $\Ad P$, and $W_k$ is an orthonormal frame for $\Hig$.
Furthermore, $\grad U_\Phi$ is the bundle gradient of $U$ evaluated at $\Phi$, satisfying 
\begin{equation*}
    \Re \langle \alpha, \grad U_\Phi \rangle = \diff U_\Phi(\alpha),
\end{equation*}
for all $\alpha \in \Gamma(\Hig)$.
Note that $\mathfrak{J}_3$ can also be defined dually as
\begin{equation}\label{eq-yukawa-complex-antilinear}
    2\Re \langle \Phi, \langle \Psi, i\Y^-  \Psi \rangle \rangle = \langle \Psi, i\Y_\Phi  \Psi \rangle
\end{equation}
for all $\Phi \in \Gamma(\Hig)$ and $\Psi \in \Gamma(\Fer_+)$.

\section{The energy-momentum tensor of the Standard Model}

The \emph{energy-momentum tensor} $T$ of a given matter model is defined as the variation of the associated Lagrangian with respect to the metric $g$ of the underlying spacetime, cf.\ \eqref{eq-energy-momentum-def} below.
It thus appears naturally as the right-hand side of the Einstein field equations
\begin{equation*}
    \Ein_g = T,
\end{equation*}
where $\Ein_g = \Ric_g - \frac12 \Scal_g g$ is the Einstein tensor. 
In particular, the energy-momentum tensor describes how matter affects the gravitational field.
From a mathematical perspective, the energy-momentum tensor is also useful for deriving analytical estimates due to the fact that it is divergence-free, as we shall verify below.

\subsection{Derivation of the energy-momentum tensor}

The energy-momentum tensor of the Standard Model is defined as the real-valued symmetric tensor field $T^{\mathrm{SM}}$ satisfying
\begin{equation}\label{eq-energy-momentum-def}
    \frac{\diff}{\diff t}\bigg|_{t=0} \int_K L^{\mathrm{SM}}_{g(t)}(\omega,\Phi,\Psi) \dv_{g(t)} 
    = \int_K \tr_h(T^{\mathrm{SM}}) \dv_g 
    = \int_K h^{\mu\nu} T^{\mathrm{SM}}_{\mu\nu} \dv_g,
\end{equation}
where $K \subset M$ is compact and $g(t)$ is a variation of the metric 
\begin{equation}
\label{eq:variation-metric}
g(0) = g, \qquad \frac{\diff}{\diff t}\bigg|_{t=0} g(t) = h, 
\end{equation}
such that $h$ is supported in $K$.

The derivation of the energy-momentum tensor for the different sectors of the Standard Model in isolation is well-known in the literature, with the only exception being the Higgs sector with inhomogeneous potentials, such as the conformal potential as detailed in \eqref{eq-conformal-potential}.
As we are treating the full Standard Model within this manuscript, we nevertheless present a detailed derivation including all of its sectors here.

\begin{Satz}
\label{thm:energy-momentum-tensors}
The energy-momentum tensor associated with the Standard Model Lagrangian
\eqref{eq-sm-lagrangian}  with the Higgs potential $U$ independent of $g$ is given at a critical point by
\begin{align}
 \label{dfn:energy-momentum}
T^{\mathrm{SM}}=T^{\mathrm{YM}}+T^{\mathrm{Dirac}}+T^{\mathrm{Higgs}},
\end{align}
where
\begin{align*}
T^{\mathrm{YM}}(X,Y)
&= \langle X \iprod F_\omega, Y \iprod F_\omega \rangle - \frac12 |F_\omega|^2 g(X,Y),
\\[0.2cm]
T^{\mathrm{Dirac}}(X,Y) 
&= \frac14 \Re \langle \Psi, iX \cdot (\nabla_\omega\Psi)(Y) + iY \cdot (\nabla_\omega\Psi)(X) \rangle,
\\[0.2cm]
T^{\mathrm{Higgs}}(X,Y)
&= \Re \langle (\nabla_\omega \Phi)(X), (\nabla_\omega \Phi)(Y) \rangle - \frac12 \left(|\nabla_\omega\Phi|^2 + U(\Phi)\right) g(X,Y),
\end{align*}
for vector fields $X,Y \in \Gamma(TM)$.
On the other hand, the energy-momentum tensor associated with the conformal Standard Model Lagrangian, i.e.\ with $U = U_{\mathrm{conf}}$ as defined in \eqref{eq-conformal-potential}, is given by 
\begin{align}
 \label{dfn:energy-momentum-conf}
T^{\mathrm{c-SM}}=T^{\mathrm{YM}}+T^{\mathrm{Dirac}}+T^{\mathrm{c-Higgs}},
\end{align}
where $T^{\mathrm{YM}}$ and $T^{\mathrm{Dirac}}$ are as above, while
\begin{align*}
    &T^{\mathrm{c-Higgs}}(X,Y)\\
    =&\, \frac{1}{2(m-1)} \bigg[ (2-m)\Re \langle \Phi, (\nabla_\omega^2\Phi)(X,Y) \rangle + m\Re \langle (\nabla_\omega \Phi)(X), (\nabla_\omega \Phi)(Y) \rangle + \frac{m-2}{2} |\Phi|^2 \Ric_g(X,Y)  \\
    & \qquad\quad  +\left(- |\nabla_\omega\Phi|^2  - \frac{m-2}{4(m-1)} \Scal_g |\Phi|^2  + \frac{m-3}{2} \lambda |\Phi|^4  + \frac{m-2}{2} \langle \Psi, i\Y_\Phi \Psi \rangle  \right)g(X,Y) \bigg].
\end{align*}
\end{Satz}

\begin{proof}
We consider a variation of the metric \eqref{eq:variation-metric}.
Recall that the cometric coefficients then satisfy
\begin{equation}\label{eq:variation-cometric}
    \frac{\diff}{\diff t}\bigg|_{t=0} g(t)^{\mu\nu} = -h^{\mu\nu},
\end{equation}
where the indices on the right-hand side are raised with respect to $g=g(0)$.
Recall also that the variation of the volume form satisfies
\begin{equation}\label{eq:variation-volume}
    \frac{\diff}{\diff t}\bigg|_{t=0}\dv_{g(t)}=\frac{1}{2} \tr_h(g) \dv_g,
\end{equation}
while the scalar curvature satisfies
\begin{align}
     \label{eq:variation-scalar}
        \frac{\diff}{\diff t}\bigg|_{t=0}\Scal_{g(t)}=
        -\Box_g(\tr_h(g))  + \nabla_g^*\nabla_g^* h  - \tr_h(\Ric_g),
\end{align}
see e.g.\ \cite[Theorem 1.174]{MR2371700} (but note the different convention for the sign of $\Box_g$).

By \eqref{eq:variation-volume} and the definition \eqref{eq-energy-momentum-def} of the energy-momentum tensor we first observe that
\begin{equation*}
    \int_K \tr_h(T^{\mathrm{SM}}) \dv_g
    = \int_K \left( \frac{\diff}{\diff t}\bigg|_{t=0} L^{\mathrm{SM}}_{g(t)}(\omega,\Phi,\Psi) + \tr_h\left(\frac12 L^{\mathrm{SM}}_g(\omega,\Phi,\Psi) g\right)  \right)\dv_{g},
\end{equation*}
so that it suffices to compute the variation of the density
\begin{equation*}
    \frac{\diff}{\diff t}\bigg|_{t=0} L^{\mathrm{SM}}_{g(t)}(\omega,\Phi,\Psi).
\end{equation*}
We will calculate the variation of each of the three sectors in the Lagrangian density \eqref{eq-sm-lagrangian} separately.

Starting with the Yang-Mills part, 
a direct calculation using \eqref{eq:variation-cometric} shows that
\begin{equation*}
    -\frac{\diff}{\diff t}\bigg|_{t=0}|F_\omega|^2
    =
    -\frac{\diff}{\diff t}\bigg|_{t=0} 
    \left( \frac12 g(t)^{\mu\nu}g(t)^{\alpha\beta}\langle F_{\mu\alpha}, 
    F_{\nu\beta} \rangle \right) 
    = h^{\mu\nu} \langle F_{\mu\alpha}, 
    \tensor{F}{_\nu^\alpha} \rangle
    =
    \tr_h \langle \cdot \iprod F_\omega, \cdot \iprod F_\omega \rangle,
\end{equation*}
which gives the energy-momentum tensor of the Yang-Mills sector, see also \cite{MR2453667}.

While there exist numerous approaches in the physics literature on the derivation of the energy-momentum tensor for the Dirac action (most of them are local in nature) we want to briefly recall
the generalized cylinder construction presented in \cite[Section 6]{MR2121740}.
This construction naturally extends the invariant geometric variation of the Dirac action from the Riemannian case \cite{MR1158762} to metrics of arbitrary signature.
Although we are considering spinors twisted by an additional vector bundle \(\mathscr{S}\) here, we can still apply the main results of \cite{MR2121740} since in our case \(\mathscr{S}\) is a bundle associated to $P$ and is in particular independent of the metric \(g\) of \(M\).
To this end, we form the generalized cylinder \(\mathcal{Z}=I\times M\) which we equip with the metric 
\(g=dt^2+g(t)\), and in the following we abbreviate \((M,g(t))\) by \(M_t\). Spin structures on \(M\) and \(\mathcal{Z}\) are in 1-1 correspondence and one can naturally relate $\Sigma M_t$ with a subbundle of $\Sigma \mathcal{Z}$ for any $t$.
For \(x\in M\) arbitrary and \(t_0,t_1\in I\) parallel translation on \(\mathcal{Z}\)
along the curve \(t\mapsto(t,x)\) is a linear isometry \(\tau_{t_0}^{t_1}\colon\Sigma_xM_{t_0}\to \Sigma_xM_{t_1}\) which is obtained by solving a system of linear ODEs of first order and allows for a natural identification of the spinor bundles $\Sigma M_t$ for different values of $t$.
A calculation then shows the following variation formula for the twisted Dirac operator (see \cite[Theorem 5.1]{MR2121740} for the precise details)
\begin{align}
\label{eq:variation-dirac-cylinder}
    \frac{\diff}{\diff t}\bigg|_{t=0}\tau_t^{0}\dirac^{g_t}_\omega\tau_{0}^t\Psi
    =&-\frac{1}{2} \eta^{\nu\alpha}\eta^{\mu\beta} h(e_\nu,e_\mu)e_\alpha\cdot(\nabla_\omega^{\Sigma M_{0}\otimes \mathscr{S}}\Psi)(e_\beta)  \\
    \nonumber&+\frac{1}{4}\operatorname{grad}^{M_{0}}(\tr_{g_{0}}(h))\cdot\Psi \\
    \nonumber&-\frac{1}{4}\operatorname{div}^{M_{0}}(h)\cdot\Psi.
\end{align}
Here,
\(\dirac^{g_t}_\omega\) represents the
twisted Dirac operator on \(M_t\) and we again emphasize that the twisting bundle \(\mathscr{S}\) is independent of the metric.
Employing the variation formula for the Dirac operator \eqref{eq:variation-dirac-cylinder}
and using the symmetry of Clifford multiplication, we get
\begin{align*}
    -&\frac{\diff}{\diff t}\bigg|_{t=0} \Re\langle\tau_0^t(\Psi), i\dirac^{g_t}_\omega \tau_0^t(\Psi)\rangle
    \\
    &=
    \frac12 \eta^{\nu\alpha}\eta^{\mu\beta} h_{\nu\mu} \, \Re \langle \Psi, ie_\alpha \cdot (\nabla_\omega\Psi)(e_\beta) \rangle - \frac14 \Re \langle \Psi, i\left( \operatorname{grad}_g(\tr_g h) - \operatorname{div}_g(h) \right) \cdot \Psi \rangle
    \\
    &=
    \frac14 h^{\alpha\beta} \, \Re \langle \Psi, ie_\alpha \cdot (\nabla_\omega\Psi)(e_\beta) + ie_\beta \cdot (\nabla_\omega\Psi)(e_\alpha) \rangle,
\end{align*}
where in the last step we also symmetrize the first term.
Furthermore, since the Yukawa-term does not depend on the metric, in the sense that it is 
independent of \(t\) after applying the isometry \(\tau_0^t\), we get
\begin{align*}
-\frac{\diff}{\diff t}\bigg|_{t=0}\Re\langle\tau_0^t(\Psi), \tau_0^t(i\Y_\Phi\Psi) \rangle 
=
-\frac{\diff}{\diff t}\bigg|_{t=0}\Re\langle \Psi, i\Y_\Phi\Psi \rangle 
=0.  
\end{align*}
It follows that $T^{\mathrm{Dirac}}$ is the symmetric tensor
\begin{align*}
    T^{\mathrm{Dirac}}(X,Y) &= \frac{1}{4}\Re \langle \Psi, iX\cdot(\nabla_\omega\Psi)(Y) + iY\cdot(\nabla_\omega\Psi)(X)\rangle - \frac12 \Re\langle \Psi, i(\dirac_\omega\Psi + \Y_\Phi\Psi)\rangle g(X,Y)
    \\
    &= \frac{1}{4}\Re \langle \Psi, iX\cdot(\nabla_\omega\Psi)(Y) + iY\cdot(\nabla_\omega\Psi)(X)\rangle,
\end{align*}
where the final line follows since $\dirac_\omega\Psi + \Y_\Phi\Psi = 0$ for stationary points.
This completes the variation of the Dirac sector with respect to the metric.

Regarding the kinetic term of the Higgs sector we find
\begin{equation*}
    -\frac{\diff}{\diff t}\bigg|_{t=0} |\nabla_\omega \Phi|^2
    =
    -\frac{\diff}{\diff t}\bigg|_{t=0} \left( g(t)^{\mu\nu} \Re\langle (\nabla_\omega \Phi)_\mu, (\nabla_\omega \Phi)_\nu \rangle \right)
    = \tr_h \Re\langle \cdot \iprod \nabla_\omega \Phi, \cdot \iprod \nabla_\omega \Phi \rangle,
\end{equation*}
where we used \eqref{eq:variation-cometric}.

If the Higgs potential does not depend on the metric we get
\begin{align*}
    \frac{\diff}{\diff t}\bigg|_{t=0} U(\Phi)=0,
\end{align*}
and the first part of the result follows.

For the conformal potential \eqref{eq-conformal-potential}, we employ \eqref{eq:variation-scalar} and find
\begin{equation*}
    -\int_K \frac{\diff}{\diff t}\bigg|_{t=0} U_{\mathrm{conf}}(\Phi) \dv_g = \frac{m-2}{4(m-1)}\int_K \tr_h\left(  \Box(|\Phi|^2) g -  \nabla^2(|\Phi|^2) + |\Phi|^2 \Ric_g \right) \dv_g.
\end{equation*}
It follows that
\begin{align}
T^{\mathrm{c-Higgs}} =&\,\Re\langle \nabla_\omega \Phi \otimes \nabla_\omega \Phi \rangle - \frac12 \left(|\nabla_\omega\Phi|^2 + \frac{m-2}{4(m-1)} \Scal_g |\Phi|^2 + \frac{\lambda}{2} |\Phi|^4 \right)g
\\ \nonumber
& +  \frac{m-2}{4(m-1)}\left( \Box(|\Phi|^2) g - \nabla^2(|\Phi|^2) + |\Phi|^2 \Ric_g \right). 
\end{align}
Now, note that
\begin{equation*}
    \frac12 \nabla^2|\Phi|^2 = \Re\langle \nabla_\omega \Phi \otimes \nabla_\omega \Phi \rangle + \Re \langle \Phi, \nabla_\omega^2\Phi \rangle
\end{equation*}
and hence
\begin{align*}
    \frac12 \Box |\Phi|^2 &= |\nabla_\omega\Phi|^2 + \Re\langle \Phi, \Box_\omega \Phi \rangle  \\
    &= |\nabla_\omega\Phi|^2 + \frac12 \Re \langle \Phi, \grad U_\Phi \rangle + \Re \langle \Phi, \langle \Psi, i\Y^- \Psi \rangle\rangle\\
    &= |\nabla_\omega\Phi|^2 + \frac{m-2}{4(m-1)} \Scal_g |\Phi|^2 + \lambda |\Phi|^4 + \frac12 \langle \Psi, i\Y_\Phi \Psi \rangle,
\end{align*}
provided that the Euler-Lagrange equation for the Higgs field \eqref{eq-higgs} is satisfied, which gives the desired formula for $T^{\mathrm{c-Higgs}}$ after rearranging the contributing terms.
\end{proof}

As to be expected the energy-momentum tensor obtained in \eqref{dfn:energy-momentum} is symmetric.

\subsection{Properties of the energy-momentum tensor}
This subsection presents a number of features of the energy-momentum tensor that was derived
in the previous subsection.

\begin{Prop}
    If $(\omega, \Phi, \Psi)$ is a solution of the SM equations {\normalfont (\ref{eq-yang-mills}--\ref{eq-dirac})} with a Higgs potential $U$ independent of $g$, then 
    \begin{align*}
        \nabla^\ast T^{\mathrm{SM}} &= 0, \\[0.1cm]
        \tr_g T^{\mathrm{SM}} &= -\frac{1}{2}(m-4)|F_\omega|^2
        -\frac{1}{2}\langle\Psi,i\Y_\Phi\Psi\rangle
        - \frac{1}{2}(m-2)|\nabla_\omega\Phi|^2 - \frac{m}{2}U(\Phi).
    \end{align*}
\end{Prop}

\begin{proof}
We first compute the trace by computing the trace of each sector separately, using Theorem \ref{thm:energy-momentum-tensors}, to get
\begin{align}
    \label{eq:trace-ym}
    \tr_g T^{\mathrm{YM}} &= -\frac{1}{2}(m-4)|F_\omega|^2, \\
    \label{eq:trace-dirac}
    \tr_g T^{\mathrm{Dirac}} &= \frac12 \Re\langle\Psi,i\dirac_\omega\Psi\rangle = -\frac{1}{2}\langle\Psi,i\Y_\Phi\Psi\rangle, \\
    \label{eq:trace-higgs}
    \tr_g T^{\mathrm{Higgs}} &= -\frac{1}{2}(m-2)|\nabla_\omega\Phi|^2-\frac{m}{2}U(\Phi).
\end{align}
The formula for $\tr_g T^{\mathrm{SM}}$ then follows by summing up all the contributions.

Next, we show that $T^{\mathrm{SM}}$ is divergence-free.
We start by calculating the divergence of the Yang-Mills contribution,
\begin{align}\nonumber
    (\nabla^\ast T^{\mathrm{YM}}) (e_\nu) 
    &= -\eta^{\lambda\mu}(\nabla T^{\mathrm{YM}})(e_\lambda, e_\mu, e_\nu)\\[0.1cm]\nonumber
    &= -\eta^{\lambda\mu}\langle(\diff_\omega^\ast F_\omega)_\lambda, F_{\mu\nu}\rangle\\\label{eq-ym-energy-momentum-divergence}
    &= \Re\langle (\nabla_\omega \Phi)^\mu, \rho_*(F_{\mu\nu}) \Phi \rangle
    - \frac12 \Im \langle e^\mu \cdot \Psi, \chi_*(F_{\mu\nu})\Psi \rangle
\end{align}
where we used the Bianchi identity in the second step and \eqref{eq-yang-mills} in the last step.

Concerning the divergence of the Dirac-contribution, we calculate
\begin{align*}
(\nabla^\ast T^{\mathrm{Dirac}}) (e_\nu) =&\,
-\eta^{\lambda\mu}(\nabla T^{\text{Dirac}})(e_\lambda, e_\mu, e_\nu)
\\
=& -\frac{1}{4}\Re\langle(\nabla_\omega\Psi)^\mu, ie_\mu\cdot(\nabla_\omega\Psi)_\nu\rangle 
-\frac{1}{4}\Re\langle(\nabla_\omega\Psi)^\mu, i e_\nu\cdot(\nabla_\omega\Psi)_\mu\rangle \\
&-\frac{1}{4}\Re\langle\Psi,i(\ferdirac{\omega}\nabla_\omega\Psi)_\nu\rangle 
-\frac{1}{4}\Re\langle\Psi,i e_\nu\cdot\Box_\omega \Psi\rangle.
\end{align*}

Now, we note that
\begin{equation*}
     \Re\langle(\nabla_\omega\Psi)^\mu, ie_\mu \cdot (\nabla_{\omega}\Psi)_\nu\rangle
     =-\Re\langle i\ferdirac{\omega}\Psi, (\nabla_{\omega}\Psi)_\nu\rangle
     =\Re\langle i\Y_\Phi\Psi, (\nabla_{\omega}\Psi)_\nu \rangle,
\end{equation*}
where we used the Euler-Lagrange equation for \(\Psi\), i.e. \eqref{eq-dirac}.
Moreover, since Clifford multiplication is symmetric \eqref{eq-clifford-mult}, multiplication by $ie_\nu$ is a skew-adjoint operator, and thus we can conclude $\Re \langle (\nabla_\omega \Psi)^\mu, ie_\nu \cdot (\nabla_\omega \Psi)_\mu \rangle = 0$.
Interchanging covariant derivatives yields
\begin{align*}
(\ferdirac{\omega}\nabla_\omega\Psi)_\nu 
&=
(\nabla_\omega\ferdirac{\omega}\Psi)_\nu 
+e^\mu\cdot R^{\Fer_+}(e_\mu,e_\nu)\Psi\\
&=
(\nabla_\omega\ferdirac{\omega}\Psi)_\nu 
+\underbrace{e^\mu\cdot R^{\Sigma M}(e_\mu,e_\nu)\Psi}_{=\frac{1}{2}\Ric^M(e_\nu)\cdot\Psi}
+ e^\mu\cdot \chi_\ast(F_{\mu\nu})\Psi,
\end{align*}
which allows us to deduce
\begin{align*}
\Re\langle\Psi, i(\ferdirac{\omega}\nabla_{\omega}\Psi)_\nu\rangle
&= \Re\langle\Psi,i(\nabla_{\omega}\ferdirac{\omega}\Psi)_\nu\rangle
+\frac{1}{2}\Re\langle\Psi, i\Ric^M(e_\nu)\cdot\Psi\rangle
+\Re\langle\Psi, ie^\mu\cdot \chi_\ast(F_{\mu\nu})\Psi\rangle
\\
&= \Re\langle\Psi,i(\nabla_{\omega}\ferdirac{\omega}\Psi)_\nu\rangle
+ \Re\langle\Psi, ie^\mu\cdot \chi_\ast(F_{\mu\nu})\Psi\rangle
\end{align*}
where the second term again vanishes since Clifford multiplication is symmetric such that $i\Ric^M(e_\nu)$ is skew-adjoint.
Using the Weitzenböck formula for the twisted Dirac operator \eqref{eq:weitzenboeck-dirac}
we can infer 
\begin{equation*}
\Re\langle\Psi, i e_\nu\cdot\Box_\omega\Psi\rangle
= -\Re\langle\Psi, i e_\nu\cdot\ferdirac{\omega}^2\Psi\rangle
+\frac{1}{4}\Scal_g
\underbrace{\Re\langle\Psi,i e_\nu\cdot\Psi\rangle}_{=0}
+\Re\langle\Psi, i e_\nu\cdot\chi_\ast(F_\omega)\cdot\Psi\rangle.
\end{equation*}
Combining the above identities we arrive at
\begin{align*}
(\nabla^\ast T^{\mathrm{Dirac}})(e_\nu)
=&-\frac{1}{4}\Re\langle i \Y_\Phi\Psi, (\nabla_{\omega}\Psi)_\nu\rangle
-\frac{1}{4}\Re\langle\Psi, i(\nabla_{\omega}\ferdirac{\omega}\Psi)_\nu\rangle \\
&+ \frac{1}{4}\Re\langle\Psi, ie_\nu\cdot\ferdirac{\omega}^2\Psi\rangle
- \frac{1}{4}\Re\langle\Psi, ie_\nu\cdot\rho_\ast(F_\omega)\cdot\Psi\rangle.
\end{align*}
Now, using the Euler-Lagrange equation \eqref{eq-dirac} for the twisted spinor \(\Psi\) we find
\begin{align*}
(\nabla_{\omega}\ferdirac{\omega}\Psi)_\nu
&=
-\Y_{(\nabla_\omega\Phi)_\nu}\Psi
-\Y_\Phi(\nabla_\omega\Psi)_\nu,\\
\ferdirac{\omega}^2\Psi
&= \Y_\Phi(\Y_\Phi\Psi)
- \Y_{\nabla_\omega\Phi} \cdot \Psi.
\end{align*}
Thus, we get
\begin{align*}
\Re\langle\Psi, i(\nabla_{\omega}\ferdirac{\omega}\Psi)_\nu\rangle
&=\Re\langle\Psi, i\Y_{(\nabla_\omega\Phi)_\nu}\Psi\rangle
+\Re\langle\Psi,i\Y_\Phi(\nabla_\omega\Psi)_\nu\rangle,\\
\Re\langle\Psi,ie_\nu\cdot\ferdirac{\omega}^2\Psi\rangle
&=\underbrace{\Re\langle\Psi, ie_\nu\cdot \Y_\Phi(\Y_\Phi\Psi)\rangle}_{=0}
-\Re\langle\Psi, ie_\nu\cdot e_\mu \cdot \Y_{(\nabla_\omega\Phi)^\mu}\Psi\rangle.
\end{align*}

Using the Clifford relations \eqref{eq-clifford-rel}, we find that
\begin{align*}
    \Re \langle\Psi, ie_\nu \cdot e_\mu \cdot \Y_{(\nabla_\omega\Phi)^\mu} \Psi \rangle
    &= \Re \langle e_\mu \cdot ie_\nu \cdot \Y_{(\nabla_\omega\Phi)^\mu} \Psi, \Psi \rangle\\
    &= -\Re \langle ie_\nu \cdot e_\mu \cdot \Y_{(\nabla_\omega\Phi)^\mu} \Psi, \Psi \rangle
    -2 \Re \langle i \Y_{(\nabla_\omega\Phi)_\nu} \Psi, \Psi \rangle,
\end{align*}
so that
\begin{equation*}
    \Re \langle\Psi, ie_\nu \cdot e_\mu \cdot \Y_{(\nabla_\omega\Phi)^\mu} \Psi \rangle
    = - \Re \langle \Psi, i \Y_{(\nabla_\omega\Phi)_\nu} \Psi \rangle,
\end{equation*}
where we have used that $\Y$ is skew-adjoint.
Hence, for the divergence of the Dirac part of the energy-momentum tensor we get
\begin{equation*}
(\nabla^\ast T^{\mathrm{Dirac}})(e_\nu)=
\frac{1}{2}\Re \langle \Psi, i\Y_{(\nabla_\omega\Phi)_\nu} \Psi \rangle
-\frac{1}{4}\Re\langle\Psi, ie_\nu\cdot\chi_\ast(F_\omega)\cdot\Psi\rangle
-\frac{1}{4}\Re\langle\Psi, ie^\mu\cdot \chi_\ast(F_{\mu\nu})\Psi\rangle,
\end{equation*}
where we used that \(i\Y_\Phi\) is self-adjoint.
Using the symmetry of Clifford multiplication \eqref{eq-clifford-mult} repeatedly shows that
\begin{equation*}
    \Re\langle\Psi, ie_\nu\cdot\chi_\ast(F_\omega)\cdot\Psi\rangle = \Re\langle \Psi, ie^\mu\cdot \chi_\ast(F_{\mu\nu})\Psi\rangle = - \Im\langle e^\mu\cdot \Psi, \chi_\ast(F_{\mu\nu})\Psi\rangle,
\end{equation*}
and hence
\begin{equation}\label{eq-dirac-energy-momentum-divergence}
(\nabla^\ast T^{\mathrm{Dirac}})(e_\nu)=
\frac{1}{2}\langle \Psi, i\Y_{(\nabla_\omega\Phi)_\nu} \Psi \rangle
+\frac{1}{2}\Im\langle e^\mu\cdot \Psi, \chi_\ast(F_{\mu\nu})\Psi\rangle.
\end{equation}

Finally, regarding the Higgs contribution of the energy-momentum tensor we find (assuming that $U$ is independent of $g$)
\begin{align}\nonumber
    (\nabla^\ast T^{\mathrm{Higgs}})(e_\nu)
    =&\, -\eta^{\lambda\mu} (\nabla T^{\mathrm{Higgs}})(e_\lambda, e_\mu, e_\nu)\\\nonumber
    =&\, - \Re\langle \tensor{(\nabla_\omega^2 \Phi)}{^\mu_\mu}, (\nabla_\omega\Phi)_\nu \rangle - \Re\langle (\nabla_\omega\Phi)_\mu, \tensor{(\nabla_\omega^2 \Phi)}{^\mu_\nu} \rangle\\\nonumber
    &+ \Re\langle (\nabla_\omega\Phi)_\mu, \tensor{(\nabla_\omega^2 \Phi)}{_\lambda^\mu} \rangle - \frac12 \diff U_\Phi ( (\nabla_\omega\Phi)_\nu )\\[0.1cm]\nonumber
    =& - \Re\langle \Box_\omega \Phi - \tfrac12 \grad U_\Phi, (\nabla_\omega\Phi)_\nu \rangle - \Re \langle (\nabla_\omega\Phi)^\mu, \rho_\ast(F_{\mu\nu})\Phi \rangle \\[0.1cm]\nonumber
    =&\, -\Re\langle \langle\Psi, i\Y^-\Psi\rangle, (\nabla_\omega\Phi)_\nu \rangle 
    -\Re\langle (\nabla_\omega\Phi)^\mu, \rho_*(F_{\mu\nu})\Phi \rangle\\
    \label{eq-higgs-energy-momentum-divergence}
    =&\, -\frac12 \langle \Psi, i\Y_{(\nabla_\omega\Phi)_\nu} \Psi \rangle -\Re\langle (\nabla_\omega\Phi)^\mu, \rho_*(F_{\mu\nu})\Phi \rangle,
\end{align}
where we have used the Higgs equation (\ref{eq-higgs}).

Now, adding up (\ref{eq-ym-energy-momentum-divergence}, \ref{eq-dirac-energy-momentum-divergence}, \ref{eq-higgs-energy-momentum-divergence}), we see that $\nabla^\ast T^{\mathrm{SM}} = 0$, as desired.
\end{proof}

\begin{Prop}
    If $(\omega, \Phi, \Psi)$ is a solution of the SM equations {\normalfont (\ref{eq-yang-mills}--\ref{eq-dirac})}, then
    \begin{align*}
        \nabla^\ast T^{\mathrm{c-SM}} &= 0, \\[0.1cm]
        \tr_g T^{\mathrm{c-SM}} &= (m-4) \left(-\frac{1}{2} |F_\omega|^2  + \frac{1}{4} \langle \Psi, i \Y_\Phi \Psi \rangle + \frac{\lambda}{4} |\Phi|^4 \right).
    \end{align*}
    In particular, $T^{\mathrm{c-SM}}$ is trace-free if the dimension $m=4$.
\end{Prop}

\begin{Bem}
    The second statement is of course expected since the system is conformally invariant in four dimensions.
\end{Bem}

\begin{proof} 
    Taking the trace  of $T^{\mathrm{c-Higgs}}$ in Theorem \ref{thm:energy-momentum-tensors} using the conformal Higgs equation \eqref{eq-higgs} we get
    \begin{equation*}
        \tr_g T^{\mathrm{c-Higgs}}
        = \frac{m-4}{4} \lambda |\Phi|^4 + \frac{m-2}{4} \langle \Psi, i \Y_\Phi \Psi \rangle.
    \end{equation*}
    The claimed formula follows by adding the above with \eqref{eq:trace-ym} and \eqref{eq:trace-dirac}.
    In particular, we see that $T^{\mathrm{c-SM}}$ is traceless when $m=4$.

    To compute the divergence, let us define
    \begin{align*}
        A &:= (2-m) \Re \langle \Phi, \nabla^2_\omega \Phi \rangle + m\Re \langle \nabla_\omega\Phi \otimes \nabla_\omega\Phi \rangle - |\nabla_\omega\Phi|^2 g,\\[0.1cm]
        B &:= \frac{m-2}{2}|\Phi|^2\Ric_g - \frac{m-2}{4(m-1)}\Scal_g|\Phi|^2g, \\[0.1cm]
        C &:= \left(\frac{m-3}{2} \lambda|\Phi|^4 + \frac{m-2}{2}\langle \Psi, i\Y_\Phi \Psi \rangle\right) g,
    \end{align*}
    so that $T^{\mathrm{c-Higgs}} = \frac{1}{2(m-1)}(A+B+C)$.
    We calculate
    \begin{equation*}
        \tensor{(\nabla \Re \langle \Phi, \nabla^2_\omega \Phi \rangle)}{^\mu_\mu_\nu}
        = \Re \left\langle (\nabla_\omega\Phi)^\mu, (\nabla^2_\omega \Phi)_{\mu\nu} \right\rangle
        + \Re \left\langle \Phi, \tensor{(\nabla^3_\omega\Phi)}{^\mu_\mu_\nu} \right\rangle, 
    \end{equation*}
    and we also have
    \begin{align*}
        \Re \left\langle \Phi, \tensor{(\nabla^3_\omega\Phi)}{^\mu_\mu_\nu} \right\rangle 
        &= \Re \left\langle \Phi,\tensor{(\nabla^3_\omega\Phi)}{^\mu_\nu_\mu} + \tensor{(\nabla_\omega (R^{\Hig}\Phi))}{^\mu_\mu_\nu} \right\rangle\\[0.1cm]
        &= \Re \left\langle \Phi,\tensor{(\nabla^3_\omega\Phi)}{_\nu^\mu_\mu} + \tensor{(R^{T^\ast M \otimes \Hig}(\nabla_\omega \Phi))}{^\mu_\nu_\mu} + \tensor{(\nabla_\omega (R^{\Hig}\Phi))}{^\mu_\mu_\nu} \right\rangle\\[0.1cm]
        &= \Re \left\langle \Phi,(\nabla_\omega \Box_\omega\Phi)_\nu - \tensor{R}{^\mu_\nu_\mu^\alpha}(\nabla_\omega \Phi)_\alpha + 2 \rho_*(F_{\mu\nu}) (\nabla_\omega \Phi)^\mu - \rho_\ast((\diff_\omega^\ast F_\omega)_\nu)\Phi \right\rangle\\[0.2cm]
        &= \Re \left\langle \Phi,(\nabla_\omega \Box_\omega\Phi)_\nu \right\rangle + \Re \left\langle \Phi, \Ric_{\mu\nu} (\nabla_\omega \Phi)^\mu\right\rangle - 2 \Re\left\langle  (\nabla_\omega \Phi)^\mu, \rho_*(F_{\mu\nu})\Phi \right\rangle,
    \end{align*}
    where we use the fact that $\rho_*$ maps into the unitary algebra.
    Next, we observe that
    \begin{align*}
        \tensor{(\nabla \Re \langle \nabla_\omega\Phi \otimes \nabla_\omega\Phi \rangle)}{^\mu_\mu_\nu} &= \Re \left\langle \Box_\omega \Phi , (\nabla_\omega\Phi)_\nu \right\rangle + \Re \left\langle (\nabla_\omega\Phi)^\mu, \tensor{(\nabla_\omega^2 \Phi)}{_\mu_\nu} \right\rangle, \\[0.1cm]
        \tensor{(\nabla (|\nabla_\omega\Phi|^2 g))}{^\mu_\mu_\nu}
        &= \diff |\nabla_\omega\Phi|^2(e_\nu)\\
        &= 2 \Re \left\langle (\nabla_\omega\Phi)^\mu, \tensor{(\nabla_\omega^2 \Phi)}{_\nu_\mu} \right\rangle\\
        &= 2\Re \left\langle (\nabla_\omega\Phi)^\mu, \tensor{(\nabla_\omega^2 \Phi)}{_\mu_\nu} \right\rangle - 2\Re \langle (\nabla_\omega\Phi)^\mu, \rho_\ast(F_{\mu\nu}) \Phi \rangle.
    \end{align*}
    It follows that if we take the divergence of $A$, all the terms containing $\tensor{(\nabla_\omega^2 \Phi)}{_\mu_\nu}$ cancel out, and we get
    \begin{align*}
        \tensor{(\nabla A)}{^\mu_\mu_\nu} 
        =&\, (2-m) \,\Re \langle \Phi, (\nabla_\omega \Box_\omega\Phi)_\nu \rangle 
        + m \Re \left\langle \Box_\omega \Phi , (\nabla_\omega\Phi)_\nu \right\rangle
        \\[0.1cm]
        &  + (2-m) \, \Re \langle \Phi, \Ric_{\mu\nu}(\nabla_\omega \Phi)^\mu \rangle
        + 2(m-1) \, \Re \langle (\nabla_\omega \Phi)^\mu, \rho_*(F_{\mu\nu}) \Phi \rangle.
    \end{align*}
    Next, we calculate
    \begin{align*}
        \tensor{(\nabla (|\Phi|^2\Ric_g))}{^\mu_\mu_\nu} 
        &= 
        (\diff|\Phi|^2)^\mu \Ric_{\mu\nu} + |\Phi|^2 \tensor{(\nabla \Ric_g)}{^\mu_\mu_\nu}\\[0.1cm]
        &= 
        2\Re \langle \Phi, \Ric_{\mu\nu} (\nabla_\omega\Phi)^\mu \rangle + \frac12 (\diff\Scal_g)_\nu |\Phi|^2,
    \end{align*}
    \begin{equation*}
        \tensor{(\nabla (\Scal_g|\Phi|^2g))}{^\mu_\mu_\nu} 
        = \diff(\Scal_g|\Phi|^2)_\nu
        = (\diff \Scal_g)_\nu |\Phi|^2 + 2 \Scal_g \Re\langle \Phi, (\nabla_\omega\Phi)_\nu \rangle,
    \end{equation*}
    where we have used the well-known relation between the divergence of the Ricci tensor and the scalar curvature.
    Then if we consider the divergence of $A+B$, we see that the terms involving the Ricci curvature cancel, and we are left with
    \begin{align} \label{eq-A+B-divergence}
        \tensor{(\nabla (A+B))}{^\mu_\mu_\nu} 
        =&\, (2-m) \,\Re \langle \Phi, (\nabla_\omega \Box_\omega\Phi)_\nu \rangle 
        + m \Re \left\langle \Box_\omega \Phi , (\nabla_\omega\Phi)_\nu \right\rangle
        \\[0.1cm] \nonumber
        & + \frac{(m-2)^2}{4(m-1)} (\diff \Scal_g)_\nu |\Phi|^2 - \frac{(m-2)}{2(m-1)} \Scal_g \Re\langle \Phi, (\nabla_\omega\Phi)_\nu \rangle\\[0.2cm]
        \nonumber
        & + 2(m-1) \, \Re \langle (\nabla_\omega \Phi)^\mu, \rho_*(F_{\mu\nu}) \Phi \rangle.
    \end{align}
    Finally, we expand the $\Box_\omega\Phi$ terms using the conformal Higgs equation. To this end, we calculate
    \begin{align*}
        \Re \langle \Phi, (\nabla_\omega(\lambda|\Phi|^2\Phi))_\nu \rangle
        &= \lambda\Re \langle \Phi, (\diff|\Phi|^2)_\nu \Phi + |\Phi|^2 (\nabla_\omega\Phi)_\nu \rangle,\\
        &= 3 \lambda |\Phi|^2\,\Re \langle \Phi, (\nabla_\omega\Phi)_\nu\rangle,\\[0.2cm]
        \Re \langle \Phi, (\nabla_\omega \langle \Psi, i\Y^-\Psi \rangle)_\nu \rangle
        &= \left(\nabla \Re \langle \Phi, \langle \Psi, i\Y^-\Psi \rangle \rangle\right)_\nu
        - \Re \langle (\nabla_\omega\Phi)_\nu, \langle \Psi, i\Y^-\Psi \rangle \rangle\\
        &= \frac12 \left(\nabla \langle \Psi, i\Y_{\Phi}\Psi \rangle\right)_\nu - \frac12 \langle \Psi, i\Y_{(\nabla_\omega\Phi)_\nu} \Psi \rangle,
    \end{align*} 
    and hence
    \begin{align*}
        \Re \langle \Phi, (\nabla_\omega \Box_\omega\Phi)_\nu \rangle 
        =&\, \frac{m-2}{4(m-1)} (\diff\Scal_g)_\nu |\Phi|^2 + \frac{m-2}{4(m-1)} \Scal_g \Re\langle \Phi, (\nabla_\omega \Phi)_\nu \rangle\\[0.2cm]
        &+ 3\lambda|\Phi|^2 \,\Re \langle \Phi, (\nabla_\omega\Phi)_\nu\rangle
        + \frac12 \left(\nabla \langle \Psi, i\Y_{\Phi}\Psi \rangle\right)_\nu - \frac12 \langle \Psi, i\Y_{(\nabla_\omega\Phi)_\nu} \Psi \rangle,
    \end{align*}
    while
    \begin{equation*}
        \Re \left\langle \Box_\omega \Phi , (\nabla_\omega\Phi)_\nu \right\rangle
        = \left(\frac{m-2}{4(m-1)} \Scal_g 
        + \lambda |\Phi|^2 \right)\Re\langle \Phi, (\nabla_\omega\Phi)_\nu \rangle + \frac12 \langle \Psi, i \Y_{(\nabla_\omega\Phi)_\nu}\Psi \rangle.
    \end{equation*}
    Inserting these identities into (\ref{eq-A+B-divergence}), we see that the terms involving $\Scal_g$ terms cancel, and we get
    \begin{align*}
        \tensor{(\nabla (A+B))}{^\mu_\mu_\nu} 
        =&\, 2(3-m)\lambda|\Phi|^2\Re\langle \Phi, (\nabla_\omega\Phi)_\nu \rangle + \frac{2-m}{2} \left(\nabla \langle \Psi, i\Y_{\Phi}\Psi \rangle\right)_\nu\\
        &+ (m-1)\langle \Psi, i\Y_{(\nabla_\omega\Phi)_\nu} \Psi \rangle + 2(m-1) \, \Re \langle (\nabla_\omega \Phi)^\mu, \rho_*(F_{\mu\nu}) \Phi \rangle.
    \end{align*}
    Finally, the divergence of $C$ cancels the first two terms above, and we conclude that
    \begin{align*}
        (\nabla^\ast T^{\mathrm{c-Higgs}})(e_\nu) 
        &= -\frac{1}{2(m-1)}\tensor{(\nabla (A+B+C))}{^\mu_\mu_\nu} \\[0.1cm]
        &=
        -\frac12 \langle \Psi, i\Y_{(\nabla_\omega\Phi)_\nu} \Psi \rangle - \Re \langle (\nabla_\omega \Phi)^\mu, \rho_*(F_{\mu\nu}) \Phi \rangle,
    \end{align*}
    which cancels with the divergences of $T^{\mathrm{YM}}$ and $T^{\mathrm{Dirac}}$, as before.
\end{proof}

\section{Energy conditions}
In this section we first recall various energy conditions which appear in the study of the Einstein equations coupled to matter fields. We will then check if the various sectors of the Standard Model satisfy or violate these conditions.
The following energy conditions \cite{MR3587813,MR1348164, MR4092298} are often
employed in mathematical general relativity:

\begin{Dfn}\label{def:energy-conditions}
    Let $(M,g)$ be a Lorentzian manifold and let $T$ be the energy-momentum tensor of a matter model.
    Then $T$ is said to satisfy the
    \begin{enumerate}
        \item (NEC) \emph{null energy condition} if $T(X,X) \geq 0$ for null $X \in \Gamma(TM)$,
        \item (WEC) \emph{weak energy condition} if $T(X,X) \geq 0$ for timelike $X \in \Gamma(TM)$,
        \item (SEC) \emph{strong energy condition} if $(T - \frac{1}{m-2} \tr_g(T) g)(X,X) \geq 0$ for timelike $X \in \Gamma(TM)$, 
        \item (DEC) \emph{dominant energy condition} if WEC is satisfied and $Z=-(X \iprod T)^\sharp$ is causal and future-pointing for future-pointing timelike $X\in\Gamma(TM)$.
    \end{enumerate}
\end{Dfn}

\begin{Bem}
Note that, provided the Einstein field equation
\begin{equation*}
    \Ric_g - \frac12 \Scal_g g = T
\end{equation*}
is satisfied, the SEC essentially demands that $\Ric_g(X,X) \geq 0$ for timelike $X \in \Gamma(TM)$, by the equivalent trace-reversed form 
\begin{equation*}
    \Ric_g = T - \frac{1}{m-2} \tr_g(T)g
\end{equation*}
of the Einstein equation.
\end{Bem}

We fix $M = \R \times \Sigma$ with the metric
\begin{equation}\label{eq-ADM-metric}
    g = -N^2 \,\diff t \otimes \diff t + \diff t \otimes \beta^\flat + \beta^\flat \otimes \diff t + \bar{g},
\end{equation}
where
\begin{itemize}
    \item $N : M \to \R$ is a smooth positive function, usually referred to as the \emph{lapse function},
    \item $\beta = \beta_t \in \Gamma(T\Sigma)$ is a one-parameter family of spatial vector fields, usually referred to as the \emph{shift vector},
    \item $\bar{g} = \bar{g}_t$ is a one parameter family of Riemannian metrics on $\Sigma$, and $\flat : \Gamma(T\Sigma) \to \Gamma(T^*\Sigma)$ is the musical isomorphism with respect to $\bar{g}$. 
\end{itemize}

\begin{Bem}
    Note that one commonly also writes the decomposed metric as \cite{MR2406669}
    \begin{equation*}
        g = -\alpha^2 \,\diff t^2 + \bar{g}_{ij} (\diff x^i + \beta^i \,\diff t)\otimes(\diff x^j + \beta^j \,\diff t),
    \end{equation*}
    where $x^i$ are coordinates on $\Sigma$. This is related to the form \eqref{eq-ADM-metric} via $N^2 = \alpha^2 - |\beta|^2$, while the shift vector $\beta$ and the metric $\bar{g} = \bar{g}_{ij} \, \diff x^i \otimes \diff x^j$ are the same in both conventions.
\end{Bem}

A unit timelike normal to $\Sigma$ is given by
\begin{equation*}
    \normal = e_0 = \frac{1}{\sqrt{N^2+|\beta|^2}} (\partial_t - \beta) = \frac{1}{\alpha} (\partial_t - \beta),
\end{equation*}
and we assume the time-orientation is such that $\normal$ is future-pointing.
Since the energy conditions are stated using causal vector fields, let us note that up to rescaling, any future-pointing causal vector field $X \in \Gamma(TM)$ can then be written as
\begin{equation*}
    X = \normal + \xi \in \Gamma(TM),
\end{equation*}
where $\xi \in \Gamma(T\Sigma)$ is such that $|\xi| \leq 1$.
Then $X$ is timelike if $|\xi| < 1$ and null if $|\xi|=1$.

\subsection{Higgs}

We decompose the covariant derivative of the Higgs field into its normal part $(\nabla_\omega \Phi)(\normal) =: \frac{\nabla_\omega\Phi}{\diff\normal}$ and the spatial part $D_\omega \Phi$ which can be viewed as a spatial form, i.e.\ a $t$ dependent form on the Cauchy hypersurfaces.
Then, we can write
\begin{equation*}
    \nabla_\omega \Phi = -\normal^\flat \otimes \frac{\nabla_\omega\Phi}{\diff\normal} + D_\omega\Phi,
\end{equation*}
where $\normal^\flat$ is the musical dual of $\normal$ satisfying $\normal^\flat(\normal) = g(\normal, \normal) = -1$ (hence the minus sign in front of the first term).
Taking a semi-orthonormal frame $e_\mu$ with $e_0 = \normal$, we could also write using index notation
\begin{equation*}
    (\nabla_\omega\Phi)_0 = (\nabla_\omega\Phi)(\normal) = \frac{\nabla_\omega\Phi}{\diff\normal},
    \qquad
    (\nabla_\omega\Phi)_i = (\nabla_\omega\Phi)(e_i) = (D_\omega \Phi)_i.
\end{equation*}
It follows that
\begin{equation*}
    |\nabla_\omega \Phi|^2 
    = \eta^{\mu\nu} \langle (\nabla_\omega \Phi)_\mu, (\nabla_\omega \Phi)_\nu \rangle
    = -\left|\frac{\nabla_\omega\Phi}{\diff\normal}\right|^2 + |D_\omega \Phi|^2,
\end{equation*}
and
\begin{equation*}
    \Re \langle \nabla_\omega \Phi \otimes \nabla_\omega \Phi \rangle
    = \left|\frac{\nabla_\omega \Phi}{\diff\normal}\right|^2 \normal^\flat \otimes \normal^\flat
    - \Re\left\langle \frac{\nabla_\omega \Phi}{\diff\normal}, \, (D_\omega\Phi)_i \right\rangle (\normal^\flat \otimes e^i + e^i \otimes \normal^\flat)
    + \Re\langle D_\omega \Phi \otimes D_\omega \Phi\rangle.
\end{equation*}
Thus, the energy-momentum tensor is given by
\begin{align*}
    T^{\mathrm{Higgs}}
    =&\,
    \frac12 \left( \left|\frac{\nabla_\omega \Phi}{\diff\normal}\right|^2 + |D_\omega\Phi|^2 + U(\Phi) \right) \normal^\flat \otimes \normal^\flat 
    \\
    &- \Re\left\langle \frac{\nabla_\omega \Phi}{\diff\normal}, \, (D_\omega\Phi)_i \right\rangle (\normal^\flat \otimes e^i + e^i \otimes \normal^\flat)
    \\
    &+ \left( \frac12\left|\frac{\nabla_\omega \Phi}{\diff\normal}\right|^2 \delta_{ij} + \Re \langle (D_\omega\Phi)_i, (D_\omega\Phi)_j \rangle - \frac12 |D_\omega \Phi|^2 \delta_{ij} - \frac12 U(\Phi) \delta_{ij}  \right) e^i \otimes e^j.
\end{align*}

Now, for a causal vector field $X = \normal + \xi$ with $|\xi| \leq 1$, we calculate
\begin{align*}
    T^{\mathrm{Higgs}}(X, X)
    =&\, \frac{1 + |\xi|^2}{2} \left| \frac{\nabla_\omega\Phi}{\diff \normal} \right|^2 
    + \frac{1-|\xi|^2}{2} \left( |D_\omega\Phi|^2 + U(\Phi) \right)\\
    &+ |(D_\omega \Phi)(\xi)|^2 + 2\Re\left\langle \frac{\nabla_\omega\Phi}{\diff\normal},\, (D_\omega\Phi)(\xi) \right\rangle. 
\end{align*}
In particular, if $|\xi|=1$, then we get
\begin{equation*}
    T^{\mathrm{Higgs}}(X, X) = \left| \frac{\nabla_\omega\Phi}{\diff \normal} + (D_\omega\Phi)(\xi) \right|^2 \geq 0,
\end{equation*}
so that the NEC is always satisfied (even regardless of the potential).
If $\xi=0$, then we recover the usual formula for the energy density
\begin{equation*}
    T^{\mathrm{Higgs}}(\normal,\normal) = \frac12 \left( \left|\frac{\nabla_\omega\Phi}{\diff \normal} \right|^2 + |D_\omega\Phi|^2 + U(\Phi) \right),
\end{equation*}
which shows that the WEC is satisfied in the normal direction assuming that the potential $U$ is non-negative.
In fact, if one assumes that a Higgs field needs to be covariantly constant (i.e.\ $\nabla_\omega\Phi = 0$) at its vacuum state, then the non-negativity of the potential is also necessary.
To see that the WEC is satisfied more generally,%
\footnote{In principle, one can more easily show the WEC is satisfied for all (unit) timelike vectors $X$ more simply by taking a spacetime frame having $\tilde{e}_0 = X$, as we shall demonstrate below.}
we can let
\begin{equation*}
    \gamma = \sqrt{\frac{1+|\xi|^2}{2}},
\end{equation*}
so that $|\xi| \leq \gamma \leq 1$.
Then, we can write
\begin{align*}
    T^{\mathrm{Higgs}}(X, X)
    =&\, \left| \gamma \, \frac{\nabla_\omega\Phi}{\diff\normal} + \frac{1}{\gamma}\, (D_\omega\Phi)(\xi) \right|^2 + \frac{1-|\xi|^2}{2} \left( |D_\omega\Phi|^2 - \frac{1}{\gamma^2}|(D_\omega \Phi)(\xi)|^2  + U(\Phi)\right).
\end{align*}
Now, it follows that
\begin{equation*}
    \frac{1}{\gamma}|(D_\omega \Phi)(\xi)|
    \leq 
    \frac{1}{|\xi|}|(D_\omega \Phi)(\xi)|
    \leq
    |D_\omega\Phi|,
\end{equation*}
and hence
\begin{equation*}
    T^{\mathrm{Higgs}}(X, X) \geq \frac{1-|\xi|^2}{2} \, U(\Phi),
\end{equation*}
so that the Higgs field satisfies the WEC if the potential $U(\Phi)$ is non-negative.

Next, note that
\begin{equation*}
    \tr_g (T^{\mathrm{Higgs}}) = -\frac{m-2}{2} |\nabla_\omega\Phi|^2 - \frac{m}{2}\, U(\Phi),
\end{equation*}
and hence for causal $X$,
\begin{align*}
    T(X,X) - \frac{1}{m-2} \tr_g(T) g(X,X)
    & = |(\nabla_\omega\Phi)(X)|^2 - \frac{1}{m-2} U(\Phi) g(X,X)
    \\
    &\geq \left| \frac{\nabla_\omega\Phi}{\diff\normal} + (D_\omega\Phi)(\xi) \right|^2 -\frac{1-|\xi|^2}{m-2} \, U(\Phi).
\end{align*}
Thus, we conclude that the Higgs field satisfies the SEC provided that $U(\Phi) \leq 0$. However, the case of a negative potential is not desirable for applications in physics as this would lead to a system without a state of lowest energy. In particular, the Mexican hat potential does not satisfy the SEC although it is of high relevance in physics. 

Finally, we have
\begin{align*}
    Z^{\mathrm{Higgs}} =& -(X\iprod T^{\mathrm{Higgs}})^\sharp
    \\
    =& \left[-\Re\langle (\nabla_\omega\Phi)(X), (\nabla_\omega\Phi)^\mu \rangle + \frac12 (|\nabla_\omega\Phi|^2 + U(\Phi)) X^\mu \right] e_\mu
    \\
    =&\, \frac12\left( \left| \frac{\nabla_\omega \Phi}{\diff\normal} \right|^2 + |D_\omega\Phi|^2 + U(\Phi) + 2\Re\left\langle \frac{\nabla_\omega\Phi}{\diff\normal},\, (D_\omega\Phi)(\xi) \right\rangle \right)\normal
    \\
    &+ \frac12 \left( -\left|\frac{\nabla_\omega\Phi}{\diff\normal}\right|^2 + |D_\omega\Phi|^2 + U(\Phi) \right)\xi - \Re\left\langle \frac{\nabla_\omega\Phi}{\diff\normal} + (D_\omega\Phi)(\xi),\, (D_\omega\Phi)^i \right\rangle e_i.
\end{align*}
We note that 
\begin{align*}
    g(Z, \normal) 
    =& \,
    - \frac12\left( \left| \frac{\nabla_\omega \Phi}{\diff\normal} \right|^2 + |D_\omega\Phi|^2 + U(\Phi) + 2\Re\left\langle \frac{\nabla_\omega\Phi}{\diff\normal},\, (D_\omega\Phi)(\xi) \right\rangle \right)
    \\
    \leq&\, 
     - \frac12\left( \left| \frac{\nabla_\omega \Phi}{\diff\normal} \right|^2 + |D_\omega\Phi|^2 + U(\Phi) \right) + \left|\frac{\nabla_\omega\Phi}{\diff\normal}\right| |D_\omega\Phi||\xi|
     \\
     <& 
     - \frac12\left[ \left(\left| \frac{\nabla_\omega \Phi}{\diff\normal} \right| - |D_\omega\Phi|\right)^2 + U(\Phi) \right] \leq 0
\end{align*}
if $U \geq 0$, so that $Z$ is future-directed.
For $\xi = 0$ we have
\begin{align*}
    |Z|^2 =& \, - \frac14 \left( \left| \frac{\nabla_\omega \Phi}{\diff\normal} \right|^2 + |D_\omega\Phi|^2 + U(\Phi) \right)^2
    +
    \sum_{i=1}^{m-1} \Re\left\langle \frac{\nabla_\omega\Phi}{\diff\normal},\, (D_\omega\Phi)^i \right\rangle^2 
    \\
    &\leq - \frac14 \left( \left| \frac{\nabla_\omega \Phi}{\diff\normal} \right|^2 - |D_\omega\Phi|^2 \right)^2 - \frac12 U(\Phi) \left( \left| \frac{\nabla_\omega\Phi}{\diff\normal} \right|^2 + |D_\omega\Phi|^2 + \frac12 U(\Phi) \right) \leq 0
\end{align*}
so that $Z$ is indeed causal.
More generally, we compute
\begin{align*}
    |Z|^2 =&\, 
    \Re\langle (\nabla_\omega\Phi)(X), \, (\nabla_\omega\Phi)^\mu \rangle \, \Re \langle (\nabla_\omega\Phi)(X), \, (\nabla_\omega\Phi)_\mu \rangle 
    \\
    & - (|\nabla_\omega\Phi|^2 + U(\Phi)) |(\nabla_\omega\Phi)(X)|^2 + \frac14 (|\nabla_\omega\Phi|^2 + U(\Phi))^2 g(X,X)
    \\
    =& \, |\nabla_\omega\Phi(X)|^2 \left[\Re \langle V, (\nabla_\omega \Phi)^\mu \rangle \, \Re \langle V, (\nabla_\omega\Phi)_\mu \rangle - |\nabla_\omega\Phi|^2 - U(\Phi) \right] \\
    &+ \frac14 (|\nabla_\omega\Phi|^2 + U(\Phi))^2 g(X,X),
\end{align*}
where we have set $V = (\nabla_\omega\Phi)(X) / |(\nabla_\omega\Phi)(X)|$.
We claim that 
\begin{equation*}
    \Re \langle V, (\nabla_\omega \Phi)^\mu \rangle \, \Re \langle V, (\nabla_\omega\Phi)_\mu \rangle - |\nabla_\omega\Phi|^2 \leq 0.
\end{equation*}
It is rather difficult to prove this by expanding $X = \normal + \xi$ as before, so we argue using a more standard technique as follows.
Let $\tilde{e}_0 = X/\sqrt{-|X|^2}$, which is future-directed timelike since $X$ is, and extend this to a spacetime frame $\tilde{e}_\mu$.
Then $V = (\nabla_\omega\Phi)(\tilde{e}_0) / |(\nabla_\omega\Phi)(\tilde{e}_0)|$, and
\begin{align*}
    |\Re \langle V, \nabla_\omega \Phi \rangle|^2 - |\nabla_\omega\Phi|^2
    =& -(\Re \langle V, (\nabla_\omega\Phi)(\tilde{e}_0) \rangle)^2 + \sum_{i=1}^{m-1} (\Re \langle V, (\nabla_\omega\Phi)(\tilde{e}_i) \rangle)^2\\
    &\, + |(\nabla_\omega\Phi)(\tilde{e}_0)|^2 - \sum_{i=1}^{m-1}|(\nabla_\omega\Phi)(\tilde{e}_i)|^2
    \\
    =& \,
    \sum_{i=1}^{m-1} \left[(\Re \langle V, (\nabla_\omega\Phi)(\tilde{e}_i) \rangle)^2 - |(\nabla_\omega\Phi)(\tilde{e}_i)|^2\right]
    \leq 0,
\end{align*}
since $V$ has unit norm and
\begin{equation*}
    \Re \langle V, (\nabla_\omega\Phi)(\tilde{e}_i) \rangle \leq |V| |(\nabla_\omega\Phi)(\tilde{e}_i)| = |(\nabla_\omega\Phi)(\tilde{e}_i)|.
\end{equation*}
It follows that 
\begin{equation*}
    |Z|^2 \leq - U(\Phi)|\nabla_\omega\Phi(X)|^2 +  \frac14 (|\nabla_\omega\Phi|^2 + U(\Phi))^2 |X|^2.
\end{equation*}
Thus $Z$ is causal if $U \geq 0$ since $X$ is timelike.

\subsection{Yang-Mills}

We decompose the curvature form $F_\omega$ into its electric part $E_\omega = \normal \iprod F_\omega$ and magnetic part $B_\omega$.
Note that $E_\omega$ is a spatial 1-form and $B_\omega$ is a spatial 2-form, i.e.\ they can be viewed as $t$-dependent forms on the Cauchy hypersurfaces.
Then, we can write
\begin{equation*}
    F_\omega = -\normal^\flat \wedge E_\omega + B_\omega,
\end{equation*}
or with respect to indices (in a frame with $e_0 = \normal$)
\begin{equation*}
    (F_\omega)_{0i} = F_\omega(\normal, e_i) = (E_\omega)_i, \qquad (F_\omega)_{ij} = F_\omega(e_i, e_j) = (B_\omega)_{ij}.
\end{equation*}
For simplicity, we will usually omit the subscript $\omega$ in index computations.
We have
\begin{align*}
    |F_\omega|^2 
    &= \frac12 \langle F_{\mu\nu}, F^{\mu\nu} \rangle
    = \frac12 \eta^{\mu\alpha}\eta^{\nu\beta} \langle F_{\mu\nu}, F_{\alpha \beta} \rangle
    \\
    &= - \delta^{ik} \langle F_{0i}, F_{0k} \rangle + \frac12 \delta^{ik}\delta^{j\ell} \langle F_{ij}, F_{k\ell}\rangle
    \\
    &=
    - |E_\omega|^2 + |B_\omega|^2,
\end{align*}
and 
\begin{align*}
    \langle \cdot \iprod F_\omega, \cdot \iprod F_\omega \rangle
    =&\,
    |E_\omega|^2 \, \normal^\flat \otimes \normal^\flat
    - \langle E_k, \tensor{B}{_i^k} \rangle \left( \normal^\flat \otimes e^i + e^i \otimes \normal^\flat \right)
    +
    \langle B_{ik}, \tensor{B}{_j^k} \rangle \, e^i \otimes e^j.
\end{align*}
Thus, the energy-momentum tensor is
\begin{align*}
    T^{\mathrm{YM}} 
    =&\,
    \frac12 \left( |E_\omega|^2 + |B_\omega|^2 \right) \normal^\flat \otimes \normal^\flat
    - \langle E_k, \tensor{B}{_i^k} \rangle \left( \normal^\flat \otimes e^i + e^i \otimes \normal^\flat \right)
    \\
    &+ \left( \frac12 |E_\omega|^2 \delta_{ij} + \langle B_{ik}, \tensor{B}{_j^k} \rangle - \frac12 |B_\omega|^2 \delta_{ij} \right) e^i \otimes e^j. 
\end{align*}

In particular, we observe that
\begin{equation*}
    T^{\mathrm{YM}}(\normal, \normal) = \frac12 \left( |E_\omega|^2 + |B_\omega|^2 \right),
\end{equation*}
which shows that the WEC is satisfied in the normal direction.
More generally, for $X=\normal+\xi$ with $0<|\xi| \leq 1$, we have that
\begin{align*}
    X \iprod F_\omega &= \normal \iprod F_\omega + \xi \iprod F_\omega\\
    &= E_\omega + \mathbf{n}^\flat \otimes E_\omega(\xi) + \xi \iprod B_\omega\\
    &= \eta \otimes E_\omega(\xi) + E_\omega + H_\omega,
\end{align*}
where we have defined
\begin{equation*}
    \eta = \mathbf{n}^\flat + \frac{\xi^\flat}{|\xi|^2}, \qquad H_\omega = -\frac{\xi^\flat}{|\xi|^2} \otimes E_\omega(\xi) + \xi \iprod B_\omega.
\end{equation*}
In particular, note that $\langle \eta \otimes E_\omega(\xi), E_\omega + H_\omega \rangle = 0$,
and
\begin{equation*}
    |H_\omega|^2 = \frac{|E_\omega(\xi)|^2}{|\xi|^2} + |\xi\iprod B_\omega|^2,
\end{equation*}
since $(\xi\iprod B_\omega) (\xi) = B_\omega(\xi,\xi) =0$.
Thus, we find
\begin{align*}
    T^{\mathrm{YM}}(X, X)
    &= |X\iprod F_\omega|^2 - \frac12 |F_\omega|^2 h(X,X)\\
    &= |\eta\otimes E_\omega(\xi)|^2 + |E_\omega + H_\omega|^2 + \frac{1-|\xi|^2}{2} (-|E_\omega|^2 + |B_\omega|^2)\\
    &= \frac{1-|\xi|^2}{|\xi|^2} |E_\omega(\xi)|^2 + \frac{1+|\xi|^2}{2}|E_\omega|^2  + 2\langle E_\omega, H_\omega \rangle + |H_\omega|^2  + \frac{1-|\xi|^2}{2} |B_\omega|^2\\
    &= \left|\gamma E_\omega + \frac{1}{\gamma}H_\omega\right|^2 + \left(\frac{1+|\xi|^2}{2} - \gamma^2 \right)|E_\omega|^2\\
    & \quad
    + \left(2 - |\xi|^2 - \frac{1}{\gamma^2}\right) \frac{|E_\omega(\xi)|^2}{|\xi|^2} 
    + \frac{1-|\xi|^2}{2} |B_\omega|^2 - \left(\frac{1}{\gamma^2} - 1\right) |\xi \iprod B_\omega|^2.
\end{align*}
Taking $\gamma^2 = (1+|\xi|^2)/2$ as before, we have $|\xi| \leq \gamma \leq 1$,
and we get
\begin{align*}
    T^{\mathrm{YM}}(X, X)
    &= \left|\gamma E_\omega + \frac{1}{\gamma}H_\omega\right|^2
    + \frac{1-|\xi|^2}{1+|\xi|^2} |E_\omega(\xi)|^2 
    + \frac{1-|\xi|^2}{2} \left( |B_\omega|^2 - \frac{1}{\gamma^2} |\xi \iprod B_\omega|^2 \right)
    \geq 0
\end{align*}
since  $|\xi \iprod B_\omega| \leq |B_\omega| |\xi|$,
such that both the WEC and the NEC are satisfied.

For the SEC we first observe using \eqref{eq:trace-ym} that
\begin{equation*}
    T^{\mathrm{YM}} - \frac{1}{m-2} \tr(T^{\mathrm{YM}}) g
    =
    \langle \cdot \iprod F_\omega, \cdot \iprod F_\omega \rangle
    -
    \frac{1}{m-2} |F_\omega|^2 g.
\end{equation*}
It follows that on causal vectors
\begin{align*}
    T^{\mathrm{YM}}&(X,X) - \frac{1}{m-2} \tr(T^{\mathrm{YM}}) g(X,X)
    \\[0.1cm]
    =&\, |X \iprod F_\omega|^2  - \frac{1}{m-2} |F_\omega|^2 g(X,X)
    \\
    =&\, \frac{1-|\xi|^2}{|\xi|^2} |E_\omega(\xi)|^2 + \frac{m-3 + |\xi|^2}{m-2} |E_\omega|^2 + 2\langle E_\omega, H_\omega \rangle + |H_\omega|^2 + \frac{1 - |\xi|^2}{m-2} |B_\omega|^2
    \\
    =&\,
    \left| \kappa E_\omega + \frac{1}{\kappa} H_\omega \right|^2 + \left( \frac{m-3 + |\xi|^2}{m-2} - \kappa^2 \right) |E_\omega|^2
    \\
    &+ 
    \left(2 - |\xi|^2 - \frac{1}{\kappa^2}\right)\frac{|E_\omega(\xi)|^2}{|\xi|^2}
    +
    \frac{1-|\xi|^2}{m-2} |B_\omega|^2 - \left(\frac{1}{\kappa^2}-1\right)|\xi\iprod B_\omega|^2.
\end{align*}
Analogously as before, we let
\begin{equation*}
    \kappa^2 = \frac{m-3 + |\xi|^2}{m-2},
\end{equation*}
which satisfies $|\xi| \leq \kappa \leq 1$, and leads to
\begin{align*}
    T^{\mathrm{YM}}&(X,X) - \frac{1}{m-2} \tr(T^{\mathrm{YM}}) g(X,X)
    \\[0.1cm]
    &=\left| \kappa E_\omega + \frac{1}{\kappa} H_\omega \right|^2 
    +
    \frac{(m-4 + |\xi|^2)(1-|\xi|^2)}{|\xi|^2(m-3 + |\xi|^2)}|E_\omega(\xi)|^2
    +
    \frac{1-|\xi|^2}{m-2} \left(|B_\omega|^2 - \frac{1}{\kappa^2}|\xi\iprod B_\omega|^2\right).
\end{align*}
Hence, the SEC is satisfied if $m \geq 4$.
In fact, note that the SEC and the WEC are equivalent for $m=4$ since the Yang-Mills Lagrangian is conformally invariant and the energy-momentum tensor is traceless.

Finally, we compute
\begin{align*}
    Z =& \, -(X \iprod T^{\mathrm{YM}})^\sharp
    \\
    =&\, 
    - \left[ \langle X\iprod F_\omega, e_\nu \iprod F_\omega \rangle \eta^{\mu\nu} - \frac12 |F_\omega|^2 X^\mu \right] e_\mu
    \\
    =&\, \frac12 \left( |E_\omega|^2 + 2 \langle E_\omega, \xi \iprod B_\omega \rangle + |B_\omega|^2 \right) \normal
    +
    \frac12 ( -|E_\omega|^2 + |B_\omega|^2 ) \xi
    \\
    &+ \langle E_\omega(\xi), (E_\omega)^i \rangle e_i - \left\langle (E_\omega + \xi \iprod B_\omega)_k, \; B^{ik} \right\rangle e_i.
\end{align*}
Note that
\begin{align*}
    g(Z, \mathbf{n}) =& - \frac12 \left( |E_\omega|^2 + 2 \langle E_\omega, \xi \iprod B_\omega \rangle + |B_\omega|^2 \right)
    \\
    \leq& - \frac12 |E_\omega|^2 + |E_\omega||B_\omega||\xi| - \frac12 |B_\omega|^2
    \\
    <&
    -\frac12 (|E_\omega|^2 - |B_\omega|)^2
    \leq 0, 
\end{align*}
since $|\xi| < 1$ for timelike $X$, such that $Z$ is always future-directed.
To show that it is causal, we note, assuming $\xi = 0$, that
\begin{align*}
    |Z|^2 =&\, -\frac14 ( |E_\omega|^2 + |B_\omega|^2 )^2 + \sum_{i=0}^{m-1} \langle E^j, B_{ij} \rangle^2 
    \\
    \leq&\,  
    -\frac14 ( |E_\omega|^2 + |B_\omega|^2 )^2 + |E|^2 |B|^2
    \\
    =& 
    -\frac14 (|E_\omega|^2 - |B_\omega|^2)^2
    \leq 0.
\end{align*}
More generally, one has
\begin{align*}
    |Z|^2 =& \, \eta^{\mu\nu}\Re \langle X \iprod F_\omega, e_\mu \iprod F_\omega \rangle \, \Re \langle X \iprod F_\omega, e_\nu \iprod F_\omega \rangle 
    - |F_\omega|^2 |X \iprod F_\omega|^2 + \frac14 (|F_\omega|^2)^2 |X|^2
    \\
    \leq& \, \frac14 (|F_\omega|^2)^2 |X|^2 \leq 0,
\end{align*}
where the estimate can be achieved in a similar way as for Higgs by taking a spacetime frame $\tilde{e}_0 = X/\sqrt{-|X|^2}$.
Thus, the Yang-Mills sector satisfies the DEC.

\subsection{Dirac}

For the Dirac energy-momentum tensor, one computes
\begin{equation*}
    T^{\mathrm{Dirac}}(\normal, \normal) = \frac12 \Re \left\langle \Psi, i\normal \cdot \frac{\nabla_\omega \Psi}{\diff\normal} \right\rangle = \frac12 \left(\Re \langle \Psi, ie_k \cdot (\nabla_\omega \Psi)^k \rangle + \langle \Psi, i\Y_\Phi \Psi \rangle \right),
\end{equation*}
where we use the Dirac equation in the last equality.
On null vectors $X$, we have
\begin{equation*}
    T^{\mathrm{Dirac}}(X,X) = \frac12 \Re \left\langle \Psi, iX \cdot (\nabla_\omega \Psi)(X) \right\rangle.
\end{equation*}
Clearly, the energy conditions are inconclusive (on the one hand, the inner product on spinors is indefinite, but on the other hand also the expressions in the identities above have indefinite character regardless of the inner product).

\subsection{Conformal Higgs}

The energy-momentum tensor for conformal Higgs fields has a considerably more complicated structure.
On null vectors $X$, it reads
\begin{align*}
    T^{\mathrm{c-Higgs}}(X,X) &= \frac{m-2}{4(m-1)} \left(-2\Re \langle \Phi, (\nabla^2_\omega\Phi)(X,X) \rangle + \frac{2m}{m-2} |(\nabla_\omega\Phi)(X)|^2 + |\Phi|^2 \Ric(X,X)\right),\\
    &= |(\nabla_\omega\Phi)(X)|^2 + \frac{m-2}{4(m-1)} \left(|\Phi|^2 \Ric(X,X) - \nabla^2 (|\Phi|^2)(X,X)\right).
\end{align*}
Then, we see that $T^{\mathrm{c-Higgs}}$ satisfies the NEC if
\begin{equation}\label{eq-c-higgs-null-condition}
    |\Phi|^2 \Ric \geq \nabla^2 (|\Phi|^2)
\end{equation}
on null vectors.
For the weak energy condition we compute
\begin{align}
    T^{\mathrm{c-Higgs}}(\normal, \normal) 
    =&\,  
    \frac{1}{2(m-1)} \bigg[ (2-m)\Re \left\langle \Phi, (\nabla_\omega^2\Phi)(\normal, \normal) \right\rangle + m \left|\frac{\nabla_\omega\Phi}{\diff\normal}\right|^2 + \frac{m-2}{2} |\Phi|^2 \Ric(\normal, \normal)  
    \nonumber \\
    & \qquad\quad  + |\nabla_\omega\Phi|^2 + \frac{m-2}{4(m-1)} \Scal_g |\Phi|^2 - \frac{m-3}{2} \lambda |\Phi|^4 - \frac{m-2}{2} \langle \Psi, i\Y_\Phi \Psi \rangle  \bigg] \nonumber\\[0.2cm]
    =&\, \frac12 \left( \left| \frac{\nabla_\omega\Phi}{\diff\normal} \right|^2 + |D_\omega \Phi|^2 \right) + \frac14\lambda|\Phi|^4 + \frac{m-2}{4(m-1)}|\Phi|^2 \,\mathrm{Ein}(\normal, \normal) 
    \label{eq-c-higgs-WEC}
    \\
    \nonumber
    &+ \frac{m-2}{2(m-1)} \Big[ D^* \Re\langle \Phi, D_\omega\Phi\rangle -\tr(\sff) \, \Re\langle \Phi, (\nabla_\omega\Phi)(\normal) \rangle \Big]
\end{align} 
where $\mathrm{Ein}_g = \Ric_g - \frac12 \Scal_g g$ is the Einstein tensor, and we also use
\begin{align*}
    \Re\left\langle \Phi, (\nabla_\omega^2\Phi)(\normal, \normal) \right\rangle 
    &= -\Re \langle \Phi, \Box_\omega\Phi \rangle + \Re \langle \Phi, \tensor{(\nabla_\omega^2\Phi)}{^i_i} \rangle\\
    &= -\Re \langle \Phi, \Box_\omega\Phi \rangle + (m-1)H \Re \left\langle \Phi, \frac{\nabla_\omega\Phi}{\diff\normal} \right\rangle - D^\ast \Re\langle \Phi, D_\omega\Phi \rangle - |D_\omega\Phi|^2.
\end{align*}
Moreover, $\sff(X,Y) = g(\mathbf{n}, \nabla_X Y)$ represents the second fundamental form with mean curvature $H = \frac{1}{m-1}\tr_g \sff$.

We conclude that, in general, it will not be possible to give a decisive answer if the energy-momentum tensor for the conformal Higgs potential will satisfy the weak energy condition. \\
The conformal Higgs potential shares many features with Brans-Dicke theory (see \cite{MR1968866} for a general introduction on this topic) which is an extension of general relativity that couples the Einstein-Hilbert action with a scalar field in a non-local way.
From a physics perspective one always expects certain problems arising
from such non-local couplings and this might lead to the fact
the energy conditions for the conformal Higgs potential do not encode
any useful information, see \cite{MR1791109} for an extended discussion on this subject.

For completeness, we want to mention that 
Callan–Coleman–Jackiw introduced a so-called \emph{improved stress–energy tensor} for a scalar field in \cite{MR261922} that enjoys conformal invariance in four
dimensions.  However, their approach assumes a flat space-time; nevertheless, it shares some similarities with the conformal Higgs potential.

The energy conditions from Definition \ref{def:energy-conditions} are usually formulated in general relativity with the goal of gaining  control over the Einstein tensor (or Ricci tensor) for solutions arising from the Einstein field equation. However, in cases where the energy-momentum tensor depends non-trivially on the metric (like for the conformal Higgs potential or more general non-minimal couplings of fields to gravity), it is not entirely clear whether these conditions should be modified. In particular, \eqref{eq-c-higgs-WEC} shows that the Einstein tensor also appears in the timelike component of the energy-momentum tensor itself, so that the notion of energy and the associated conditions might need to be reformulated to take this into account.

\subsection{Summary and outlook}
We summarize our results on the validity of the different energy conditions for the individual sectors of the Standard Model in the following tabular:

\par\medskip

\begin{center}
    \begin{tabular}{c|c|c|c|c}
         & NEC & WEC & SEC & DEC 
       \\ \hline
       Yang-Mills  & \ding{51} & \ding{51} & \ding{51} & \ding{51}
       \\ \hline
       Higgs  & \ding{51} \, if \(U(\Phi)\geq 0\) & \ding{51} \, if \(U(\Phi)\geq 0\) & \ding{51} \, if $U(\Phi) \leq 0$ & \ding{51} if $U(\Phi) \geq 0$
       \\  \hline
       Dirac   & \ding{55} & \ding{55} & \ding{55} & \ding{55}
       \\ \hline
       Conformal Higgs & \ding{55} & \ding{55} & \ding{55} & \ding{55}
    \end{tabular}
\end{center}
\par\medskip
When our considerations led to an inconclusive statement, we have used the symbol \ding{55} to emphasize that the condition will not be satisfied in general. Nevertheless, even if a condition is inconclusive in general, one might still have restricted classes of fields that satisfy the condition, e.g.\ a conformal Higgs field will satisfy the null energy condition if \eqref{eq-c-higgs-null-condition} is satisfied everywhere.

In particular, we observe that the Yang-Mills and the classical Higgs sector without a potential both satisfy all of the energy conditions.
This explains why these two sectors are most represented in the literature on general relativity. In these cases the energy conditions combined with the energy-momentum tensor being divergence-free can be used to derive integrated energy estimates which are useful for analytic purposes. This is also why having a non-zero potential $U$, and/or couplings to the Dirac sector usually introduces a lot of difficulties in the analysis.
Nevertheless, we hope that the presented calculations might be useful for further applications in general relativity when the Einstein equations are coupled to the matter content of the Standard Model.

\bibliographystyle{plain}
\bibliography{mybib}

\end{document}